\documentclass[10pt]{article}
\usepackage[utf8]{inputenc}
\usepackage[spanish]{babel}
\usepackage{lipsum}
\usepackage[nottoc]{tocbibind}
\usepackage{imakeidx} 
\indexsetup{othercode=\small}
\makeindex[program=makeindex,columns=2,options={-s index_style.ist},intoc]
\usepackage{url}
\usepackage[pdftex,hyperindex]{hyperref} 
\usepackage{amsmath,amsfonts,amstext,amssymb, pifont, stmaryrd,mathabx, relsize}
\usepackage{graphicx}
\usepackage[all,2cell]{xy}
\usepackage[toc]{appendix}
\usepackage{geometry}
\usepackage{nomencl}
\makenomenclature
\usepackage{amsthm}
\usepackage{thmtools}
\usepackage[dvipsnames]{xcolor}
\usepackage{tikz-cd}
\usepackage{tikz}
\usetikzlibrary{positioning, arrows}
\usepackage{comment}

\newtheorem{theorem}{Teorema}[section]
\newtheorem{proposition}[theorem]{Proposición}

\newtheorem{corollary}[theorem]{Corolario}
\theoremstyle{definition}
\newtheorem{definition}[theorem]{Definición}
\theoremstyle{remark}
\newtheorem{remark}[theorem]{Obervación}
\theoremstyle{remark}
\newtheorem{example}[theorem]{Ejemplo}

\newcommand{\field}{F}
\newcommand{\dg}{\text{DG}}

\newcommand{\dga}{\text{DGA}}
\newcommand{\dgacat}{\text{DGA-Mod}}

\begin{document}
\title{Operads libres sobre módulos diferenciales graduados}
\author{Jesús Sánchez Guevara}
\date{}
\maketitle
	\begin{abstract}
	En este artículo se aborda una construcción del funtor de operads libres,
	tanto para el caso simétrico, como para el caso no simétrico.
	Para hacer esto, los operads son interpretados como monoides  
	sobre la categoría de módulos diferenciales graduados.
	Además, se muestran algunas propiedades entre los 
	funtores asociados a este tipo de construcción 
	que ayudan a la	comprensión de sus mecanismos. 	
	\end{abstract}
\section{Introducción}

La noción de operad aparece en los años sesenta cuando James Dillon Stasheff 
usa $A_\infty$-álgebras para estudiar estructuras algebraicas que tienen propiedades asociativas no estrictas
(\cite{10.2307/1993608} y \cite{10.2307/1993609}), es decir, que se cumplen bajo equivalencias de homotopía.
Luego, con el fin de describir la estructura homotópica subyacente a los espacios de lazos,
Peter May introduce en los inicios de los años setenta el concepto de operad (\cite{may2007}),
el cual, entre otras cosas, engloba las álgebras estudiadas por Stasheff.
En su trabajo, May demuestra que, si un espacio tiene una estructura de álgebra sobre ciertos tipos de operads,
entonces este espacio es homotópicamente equivalente a un espacio de lazos, iterado finita o infinitamente.

Intuitivamente, el uso de un operad puede verse como una generalización de la relación entre un grupo y sus
representaciones, pero con la diferencia de que la estructura algebraica a describir puede que tenga muchas
operaciones de diferentes aridades y que estas operaciones no necesariamente cumplan con condiciones de
conmutatividad o asociatividad. Así, un operad $\mathcal{P}$ está conformado por diferentes tipos de objetos,
cuya naturaleza depende de la categoría en la que se esté trabajando. Estos objetos son los que organizan
las operaciones abstractas que luego serán realizadas como las operaciones que queremos describir.
Por ejemplo, $P(m)$ representaría las operaciones de nuestra entidad que tienen $m$ entradas y una salida.
Las condiciones que se le piden a estas familias de operaciones es que cuenten con una
composición asociativa y con un elemento unidad con respecto a esta operación. En el caso de los operads simétricos, 
también se necesitan condiciones de equivarianza con respecto a las acciones de los grupos simétricos.

Para efectos de este artículo, la teoría de operads será estudiada en el marco de las categorías de
módulos diferenciales graduados y se utilizará la definición clásica de operads formulada por May.
Luego, en la sección 4 se hará una descripción detallada de la construcción de operads libres simétricos y no simétricos,
la cual se puede interpretar como una generalización de las estructuras libres usuales asociadas a grupos, anillos,
espacios vectoriales, etc. 

Para realizar la construcción de operads libres se recurre a un punto de vista categórico, donde los operads
son descritos como cierto tipo de monoides. Esto nos lleva a tener una construcción libre más cercana a la
intuición, pues con los productos tensoriales asociados, las interpretaciones de los objetos como operaciones 
de varias entradas son más claras.
En la última parte se estudian las principales adjunciones entre los funtores libres de operads libres simétricos
y no simétricos, con la idea de describir el caso simétrico a través del caso no simétrico. 

El presente trabajo está basado en el segundo capítulo de mi tesis doctoral (\cite{sanchez-jesus-thesis-2016}), 
donde esta manera de abordar la construcción de operads libres lleva al diseño de operads del tipo
$E_\infty$, los cuales son usados para estudiar las propiedades homotópicas de estructuras asociadas
a complejos de cadenas, como las descritas por Alain Prouté en \cite{alain-transf-EM-1983} y \cite{alain-transf-AW-1984}.

\section{Preliminares}
		
	Se utilizarán libremente el lenguaje y las propiedades de los módulos diferenciales graduados
	que aparecen en $\cite{loday2012algebraic}$ y $\cite{sanchez-jesus-thesis-2016}$.
	Con respecto a la teoría de categorías, las terminologías y propiedades son tomadas 
	de $\cite{MacLane-1998}$ y $\cite{alain-logique-cat}$.	

	Usaremos $F$ para denotar un cuerpo, el cual puede ser 
	$\mathbb{Z}/p\mathbb{Z}$, donde $p$ es un número primo, o incluso $\mathbb{Q}$.
	Los módulos sobre $F$ serán llamados simplemente módulos.
	Se denota $[n]$ al conjunto $\{1,\ldots,n\}$, donde $n$ es un entero positivo.
	El grupo simétrico formado por las permutaciones de $[n]$ se escribe $\Sigma_n$.

	Un módulo $M$ se dirá módulo graduado si existe una familia $\{M_i\}_{i\in \mathbb{Z}}$
	de submódulos de $M$, tales que $M=\bigoplus_{i\in \mathbb{Z}}M_i$.	
	Un módulo diferencial graduado o $\dg$-módulo 
	es un módulo graduado $M$ junto con un morfismo $\partial:M\to M$
	de grado $-1$, tal que $\partial \circ \partial=0$.	
	Una aumentación de un $\dg$-módulo es un morfismo de $\dg$-módulos 
	$\epsilon:M\to \field$ de grado $0$. Similarmente, una coaumentación de $M$
	es un morfismo de $\dg$-módulos $\eta:\field\to M$ de grado $0$.
	Si un $\dg$-módulo $M$ tiene una aumentación $\epsilon$ y una coaumentación $\eta$,
	tal que $\epsilon \circ \eta=1_\field$, entonces $M$ se dice módulo diferencial
	graduado con aumentación o $\dga$-módulo. Se denota la categoría de $\dga$-módulos 
	como $\dgacat$. 

\section{Operads simétricos}

	Los operads simétricos u operads, son colecciones de objetos que tienen asociados a sus componentes
	acciones de los grupos simétricos $\Sigma_n$.
	Debido a que se trabajará en un contexto graduado, 
	los signos de los diagramas incluidos están determinados por la convención de Koszul y
	son omitidos para facilitar la lectura. Recuerde que la convención de Koszul dice que, 
	si la posición de dos símbolos en una expresión, 
	de grados $p$ y $q$, es permutada,
	la expresión resultante será multiplicada por $(-1)^{pq}$.

\begin{definition}
\label{df-operad}
\index{operad@Operad}
	Un operad $\mathcal{P}$ es una colección de $\dga$-módulos $\{\mathcal{P}(n)\}_{n\geq 0}$
	junto con:
	\begin{enumerate}
	\item Un morfismo $\eta:\field\to \mathcal{P}(1)$, llamado la unidad de $\mathcal{P}$. 
	\item Para cada $n$, una acción a la derecha por el grupo simétrico $\Sigma_n$ sobre $\mathcal{P}(n)$,
	es decir, un morfismo de $\dga$-módulos que hace de $P(n)$ un DGA-$\field[\Sigma_n]$-módulo derecho.
		\begin{equation}
			\begin{gathered}
			P(n)\otimes \field[\Sigma_n] \longrightarrow P(n)
			\end{gathered}
		\end{equation}
	\item Por cada tupla $(h,i_1,\ldots,i_h)$, un morfismo de $\dga$-módulos,
		\begin{equation}
			\begin{gathered}
			\gamma_{(h,i_1,\ldots,i_h)}: \mathcal{P}(h)\otimes \mathcal{P}(i_1)\otimes 
			\cdots \otimes \mathcal{P}(i_h)\to \mathcal{P}(n)
			\end{gathered}
		\end{equation}
	donde $n=i_1+\cdots +i_h$ y $n,h,i_j\geq 0$. Todo morfismo de este tipo se escribirá $\gamma$.
	\end{enumerate}
	Estos morfismos deben de cumplir las siguientes condiciones.
	\begin{enumerate}
	\item Los morfismos $\gamma$ son asociativos, en el sentido del siguiente diagrama conmutativo.
		\begin{equation}
		\label{dg-associ-operad}
		\begin{gathered}
		\xymatrix{
		P(h)\otimes \left[ \bigotimes_{p=1}^{h}P(i_p) \right]\otimes \left[\bigotimes_{p=1}^{h}\bigotimes_{q=1}^{i_p}P(r_{p,q})
		\ar[r]^-{\gamma \otimes 1}\right] \ar[dd]_-{\text{intercambio}} 
		& P(n)\otimes \left[ \bigotimes_{p=1}^{h}\bigotimes_{q=1}^{i_s}P(r_{p,q})\right]
		\ar[d]^-{\gamma}\\
		& P(r)\\
		P(h)\otimes \bigotimes_{p=1}^{h} \left[P(i_p)\otimes \bigotimes_{q=1}^{i_p}P(r_{p,q})\right]
		\ar[r]_-{1\otimes \gamma^{\otimes h}} 
		& P(h)\otimes \bigotimes_{p=1}^{h}P(r_p)  \ar[u]_-{\gamma}
		}
		\end{gathered}
		\end{equation}
	Donde $n=\sum_{p=1}^{h}i_p$, $r=\sum_{p=1}^{h}\sum_{q=1}^{i_p}r_{p,q}=\sum_{p=1}^{h}r_p$
	y la flecha vertical izquierda es solamente un intercambio de factores.
	\item La unidad $\eta:\field  \to P(1)$ hace conmutativos los siguientes diagramas.
		\begin{equation}
		\label{dg-unit-operad}
		\begin{gathered}
		\xymatrix@R=25pt{
		\mathcal{P}(n)\otimes \field^{\otimes n} \ar[r]^-{\cong} \ar[d]_-{1\otimes \eta^{\otimes n}}& \mathcal{P}(n)\\
		\mathcal{P}(n)\otimes \mathcal{P}(1)^{\otimes n} \ar[ru]_-{\gamma} &
		}
		\xymatrix{
		\field\otimes \mathcal{P}(n) \ar[r]^-{\cong} \ar[d]_-{\eta \otimes 1}& \mathcal{P}(n)\\
		\mathcal{P}(1)\otimes \mathcal{P}(n)\ar[ru]_-{\gamma} &
		}
		\end{gathered}
		\end{equation}
	\item La acción de los grupos simétricos deben satisfacer condiciones de equivarianza,
	expresadas por los siguientes diagramas conmutativos.
		\begin{equation}
		\label{dg-equiv-1-op}
		\begin{gathered}
		\xymatrix@C=5pc{
		\mathcal{P}(h)\otimes\mathcal{P}(i_1)\otimes \cdots \otimes\mathcal{P}(i_h)
		\ar[d]_-{\sigma \otimes \sigma^{-1}} \ar[r]^-{\gamma}
		& P(n)\ar[d]^-{\sigma(i_{1},\ldots,i_{n})}\\
		\mathcal{P}(h)\otimes\mathcal{P}(i_{\sigma(1)})\otimes \cdots \otimes \mathcal{P}(i_{\sigma(h)})
		\ar[r]^-{\gamma} & P(n)\\
		}
		\end{gathered}
		\end{equation}
	Donde $n=i_1+\cdots +i_h$ y la flecha $\sigma\otimes \sigma^{-1}$ consiste en la acción derecha de
	$\sigma$ sobre $P(h)$ y la acción izquierda de $\sigma^{-1}$ sobre el producto tensorial
	$P(i_1)\otimes \cdots \otimes P(i_h)$.
		\begin{equation}\label{dg-equiv-2-op}
		\begin{gathered}
		\xymatrix@C=5pc{
		\mathcal{P}(h)\otimes \mathcal{P}(i_1)\otimes \cdots \otimes\mathcal{P}(i_h)
		\ar[r]^-{1\otimes \tau_1\otimes \cdots \otimes \tau_h} \ar[d]_-{\gamma} & 
		\mathcal{P}(h)\otimes\mathcal{P}(i_1)\otimes \cdots \otimes\mathcal{P}(i_h)
		\ar[d]^-{\gamma}\\
		\mathcal{P}(n)\ar[r]^-{\tau_1\oplus \cdots \oplus \tau_n} & \mathcal{P}(n)
		}
		\end{gathered}
		\end{equation}
	Donde $n=i_1+\cdots +i_h$ y la acción $1\otimes \tau_1\otimes \cdots \otimes \tau_h$ 
	es la identidad de $P(h)$ sobre el primer factor y la acción derecha de $\tau_j$ en el factor $P(i_j)$.
	\end{enumerate}
\end{definition}

	Un operad fundamental que se puede ver como el paradigma para diseñar el concepto de
	operad, es el operad de endomorfismos.
\begin{definition}
\label{df-endo-operad}
\index{endomorphismoperad@Endomorphism Operad}
	Para cada $M\in \dgacat$, el operad $End(M)$ de endomorfismos de $M$ se define como:
	\begin{enumerate}
	\item Para todo $n\geq 0$, $End(M)(n)=Hom(M^{\otimes n},M)$, es decir el $\dga$-módulo de
	aplicaciones homogéneas de  $M^{\otimes n}$ a $M$.
	\item La unidad $\eta:\field \to End(M)(1)$ es la identidad de $M$. 
	\item La acción derecha de $\Sigma_n$ sobre $End(M)$ es inducida por la acción izquierda de $\Sigma_n$ 
	sobre $M^{\otimes n}$, es decir,
	$f\sigma(x_1\otimes \cdots \otimes x_n)=
	\|\sigma\|f(x_{\sigma^{-1}(1)}\otimes \cdots \otimes x_{\sigma^{-1}(n)})$. Donde $\|\sigma\|$
	es el signo de la permutación $\sigma$.
	\item La composición de aplicaciones,
		\begin{equation}
		\begin{gathered}
		\gamma:End(M)(h)\otimes End(M)(i_1)\otimes \cdots \otimes End(M)(i_h)\to End(M)(n)
		\end{gathered}
		\end{equation}
		donde $n=i_1+\cdots +i_h$, está dada por,
		\begin{equation}
		\begin{gathered}
		\gamma (f_h\otimes f_{i_1}\otimes \cdots \otimes f_{i_h})=f_h \circ(f_{i_1}\otimes \cdots \otimes f_{i_h})
		\end{gathered}
		\end{equation}
	\end{enumerate}
\end{definition}

El lector puede verificar que $End(M)$ satisface
las condiciones de la definición \ref{df-operad}.
\begin{definition}
\label{df-operads-morphism}
\index{operadmorphism@Operad morphism}
\label{df-category-operads}
\index{categoryoperads@Category of operads}
	Sean $P$ y $Q$ dos operads. 
	Un morfismo $f$ de $P$ a $Q$, es una colección de morfismos de $\dga$-módulos,
		\begin{equation}
		\begin{gathered}
		f_n:P(n)\to Q(n)
		\end{gathered}
		\end{equation}
	que satisfacen las siguientes condiciones.
	\begin{enumerate}
	\item El morfismo $f_1:P(1)\to Q(1)$ preserva la unidad de los operads, es decir, $f_1\eta=\eta$.
		\begin{equation}
		\begin{gathered}
		\xymatrix{
		P(1)\ar[rr]^-{f_1} && Q(1)\\
		&\field \ar[ul]^-{\eta} \ar[ur]_-{\eta}& 
		}
		\end{gathered}
		\end{equation}
	\item Los morfismos $f_n:P(n)\to Q(n)$ son $\Sigma_n$-equivariantes, es decir, el siguiente
	diagrama es conmutativo para cada $\sigma \in \Sigma_n$.
		\begin{equation}
		\begin{gathered}
		\xymatrix{
		P(n) \ar[r]^-{f_n} \ar[d]_-{\sigma}& Q(n)\ar[d]^-{\sigma}\\
		P(n) \ar[r]^-{f_n} & Q(n)\\
		}
		\end{gathered}
		\end{equation}
	\item $f$ preserva las operaciones de composición de los operads, 
	esto es, que el siguiente diagrama es conmutativo.
		\begin{equation}
		\begin{gathered}
		\xymatrix@C=3cm@R=1cm{
		P(h)\otimes P(i_1)\otimes \cdots \otimes P(i_h) \ar[r]^-{\gamma_P} 
		\ar[d]_-{f_h\otimes f_{i_1}\otimes \cdots \otimes f_{i_h}}
		& P(n) \ar[d]^-{f_n}\\
		Q(h)\otimes Q(i_1)\otimes \cdots \otimes Q(i_h) \ar[r]^-{\gamma_Q} 
		& Q(n) \\
		}
		\end{gathered}
		\end{equation}
	\end{enumerate}
La categoría de operads sobre $\dgacat$ se denota $\mathcal{OP}$.
\end{definition}


\begin{example}
\label{N-operad}
\index{operadn@Operad $\mathcal{N}$}
	El operad $\mathcal{N}$ está dado por $\mathcal{N}(n)=\field$ para cada $n$ no negativo,
	donde $\field$ es visto como un $\dga$-módulo concentrado en grado cero.  
	La unidad $\eta$ es la identidad de $\field$, $\Sigma_n$ actúa de manera trivial en
	cada componente y las composiciones $\gamma$, si denotamos $a_i$ al generador de grado cero de $N(i)$,
	están dadas por la regla,
		\begin{equation}
		\begin{gathered}
		\gamma:a_h\otimes a_{i_1}\otimes \cdots \otimes a_{i_h}\to a_{n}
		\end{gathered}
		\end{equation}
		donde $n=i_1+\cdots +i_h$. 
\end{example}

\begin{example}
\label{M-operad}
\index{operadm@Operad $\mathcal{M}$}
	Al hacer libre la acción de los grupos simétricos en el ejemplo anterior, se produce un operad 
	que denotamos $\mathcal{M}$. Las componentes de $\mathcal{M}$ son los
	módulos concentrados en grado cero $M(n)=\field[\Sigma_n]$,
	para cada $n$ no negativo. 
	Las composiciones $\gamma$ son determinadas, al igual que antes,
	por los generadores de grado cero, pero respetando las acciones de los grupos simétricos:
		\begin{equation}
		\begin{gathered}
		\gamma(a_h\otimes a_{i_1}\sigma_{i_1} \otimes \cdots \otimes  a_{i_h}\sigma_{i_1})= 
		a_{n}(\sigma_{i_1}\oplus \cdots \oplus \sigma_{i_h}) 
		\end{gathered}
		\end{equation}
	y
		\begin{equation}
		\begin{gathered}
		\gamma(a_h \sigma \otimes a_{i_{\sigma^{-1}(1)}}\otimes \cdots \otimes a_{i_{\sigma^{-1}(h)}})= 
		a_{n} \sigma(i_1,\ldots, i_h)
		\end{gathered}
		\end{equation}
	donde $n=i_1+\cdots +i_h$. 
\end{example}
\section{Operads libres}
	Al dejar de lado la composición en un operad $\mathcal{P}$, 
	lo que resta es una colección $\{ P(n)\}_{n\geq 0}$ de $\dga$-módulos
	con acciones a la derecha por los grupos simétricos respectivos.
	Las colecciones de este tipo las llamaremos $\mathbb{S}$-módulos.	
	En esta sección veremos una manera de generar a partir de un $\mathbb{S}$-módulo,
	una estructura de operad que lo contenga.
	Los operads obtenidos de esta forma
	son llamados operads libres y van a satisfacer una propiedad universal: 	
	cada morfismo de $\mathbb{S}$-módulos, entre la colección de partida
	a cualquier otra colección con una estructura de operad,
	puede ser extendido de manera única a un morfismo de operads desde el operad libre.

	La construcción de operads libres es largamente descrita 
	en las principales referencias del tema (por ejemplo 
	\cite{loday2012algebraic}, \cite{markl2007operads}, \cite{2005math......2155B} y \cite{rezkthesis1996}).
	Aunque hay muchas formas de definir operads libres,
	en esta presentación se tratará de mantener los operads tan cerca como se pueda de la definición clásica
	(vea \ref{df-operad}).
	Este punto de vista se puede usar, entre otras cosas, 
	para facilitar el diseño de algunos tipos de $E_\infty$-coalgebras, 
	las cuales son coalgebras sobre un operad simétrico del tipo $E_\infty$
	(ver \cite{sanchez-jesus-thesis-2016}).

	La construcción de un operad libre a partir de un $\mathbb{S}$-módulo se basa en interpretar los operads  
	como monoides sobre
	la categoría de $\mathbb{S}$-módulos, donde el producto
	monoidal se busca que codifique la composición generalizada de operaciones abstractas.
	Esta manera de definir operads es una instancia de un punto de vista mucho más general, donde los operads,
	son definidos como monads sobre la categoría de endofunctores
	de la categoría de $\dga$-módulos y donde los $\mathbb{S}$-módulos son identificados
	con funtores de Schur (ver \cite{loday2012algebraic}). 
	Aquí vamos a mantener a los $\mathbb{S}$-módulos en su estado natural
	y la composición de funtores aparecerá como una operación especial de $\mathbb{S}$-módulos.
\begin{definition}
\label{df-S-groupoid}
\index{groupoidS@Groupoid $\mathbb{S}$}
	Sea $\mathbb{S}$ el grupoide donde los objetos
	son los conjuntos ordenados $[n]=\{1,\ldots,n\}$, donde $n$ es un entero positivo
	y $[0]=\emptyset$. Los morfismos de $\mathbb{S}$ 
	están dados por $\mathbb{S}\mathbb(n,m)=\emptyset$,
	si $n\neq m$, y $\mathbb{S}(n,n)=\Sigma_n$, el $n$ grupo simétrico.
\end{definition}

\begin{definition}
\label{df-S-module}
\index{smodule@$\mathbb{S}$-module}
\index{categorysmod@Category $\mathbb{S}$-Mod}
	Un $\mathbb{S}$-módulo $M$ es un funtor contravariante de la 
	categoría $\mathbb{S}$ a la categoría $\dga$-Mod. 
	Los morfismos $\mathbb{S}(n,n)$ son interpretados
	como acciones una acción a la derecha de $\Sigma_n$ sobre $M(n)$.
	La categoría de $\mathbb{S}$-módulos y transformaciones naturales
	se denota $\mathbb{S}$-Mod.
\end{definition}
Observe que la categoría $\mathbb{S}$-Mod tiene todos los colímites y límites debido a que
es una categoría de diagramas sobre $\dga$-Mod.
\begin{definition}
\label{df-forget-functor-seq}
\index{forgetfulfunctoruopsmod@Forgetful functor $U:\mathcal{OP}\to \mathbb{S}$-Mod}
	Se denota $U$ el funtor de olvido de la categoría de operads a la categoría $\mathbb{S}$-Mod.
\end{definition}

	Antes de la formalidad de la construcción del funtor que define los operads libres sobre $\mathbb{S}$-Mod,
	veamos en detalle como luce un operad libre. 
	Sea $M$ un $\mathbb{S}$-módulo,
	si para cada entero $n\geq 1$, su componente $M(n)$ se piensa como constituido
	por aplicaciones de $n$ entradas y una salida, entonces
	el operad libre $F(M)$ asociado a $M$ se puede entender como 
	todas las posibles aplicaciones que se obtienen al componer formalmente los elementos de $M$.
	Es claro que, como $\mathbb{S}$-módulo, $F(M)$ deberá contener a $M$ como un $\mathbb{S}$-submódulo. 

	Así, la componente $F(M)(n)$ contendrá a $M(n)$ y además, 
	deberá contener a todos los posibles productos tensoriales del tipo 
	$M(h)\otimes M(i_1)\otimes \cdots \otimes M(i_h)$,
	donde $i_1+\cdots +i_h=n$, 
	ya que ellos representan las composiciones formales que dan como resultado operaciones de aridad $n$.

	Para satisfacer el axioma de equivarianza \ref{dg-equiv-1-op} 
	necesitamos, para cada $\sigma \in \Sigma_h$, la relación,
		\begin{equation}
		\begin{gathered}
		M(h)\sigma \otimes M(i_1)\otimes \cdots \otimes M(i_h)
		= M(h)\otimes M(i_{\sigma(1)})\otimes \cdots \otimes M(i_{\sigma(h)}),
		\end{gathered}
		\end{equation}
	la cual se obtiene tomando el producto tensorial sobre $\field[\Sigma_h]$,
		\begin{equation}\label{fm-equi-ope-sigmak1}
		\begin{gathered}
		M(h)\otimes_{\Sigma_h} M(i_1)\otimes \cdots \otimes M(i_h).
		\end{gathered}
		\end{equation}
	Ahora, considere el segundo axioma de equivarianza $\ref{dg-equiv-2-op}$. 
	En la parte derecha de \ref{fm-equi-ope-sigmak1}, $M(i_1)\otimes \cdots \otimes M(i_h)$,
	podríamos tener acciones sobre cada factor por elementos del respectivo
	grupo simétrico,
		\begin{equation}
		\begin{gathered}
		M(i_1)\tau_{1}\otimes \cdots \otimes M(i_h)\tau_{h} \text{, donde $\tau_j\in \Sigma_{i_j}$}
		\end{gathered}
		\end{equation}
	La acción simultánea de las permutaciones $\tau_j$ puede verse como una acción de la permutación de $\Sigma_n$
	dada por $\tau_1\oplus \cdots \oplus \tau_k$, actuando a la derecha de $M(i_1)\otimes \cdots \otimes M(i_h)$. 
	Las permutaciones de este tipo forman el subgrupo $\Sigma_{i_1}\times \cdots \times \Sigma_{i_h}$ de $\Sigma_n$.
	Entonces ponemos escribir,
		\begin{equation}
		\begin{gathered}
		M(i_1)\tau_{1}\otimes \cdots \otimes M(i_h)\tau_{h}=
		\left( M(i_1)\otimes \cdots \otimes M(i_h) \right) (\tau_1 \oplus \cdots \oplus \tau_h)
		\end{gathered}
		\end{equation}
	El proceso de colocar a la derecha las permutaciones $\tau_j$ es entonces expresado por el producto tensorial,
		\begin{equation}
		\begin{gathered}
		\left( M(i_1)\otimes \cdots \otimes M(i_h) \right)\otimes_{\Sigma_{i_1}\times \cdots \times \Sigma_{i_h}} 			\field[\Sigma_n]
		\end{gathered}
		\end{equation}
	En esta expresión se coloca $\field[\Sigma_n]$ en lugar de $\field[\Sigma_{i_1}\times \cdots \times \Sigma_{i_h}]$,
	para así considerar todas las otras permutaciones de $\Sigma_n$ que
	actúan sobre $i_1+\cdots+i_h$ entradas pero no pueden ser expresadas por
	una suma del tipo $\tau_1 \oplus \cdots \oplus \tau_h$. Además, esto se simplifica escribiendo,
		\begin{equation}\label{fm-tdlsf}
		\begin{gathered}
		\left( M(i_1)\otimes \cdots \otimes M(i_h) \right)\otimes 
		\field[\Sigma_n /(\Sigma_{i_1}\times \cdots \times \Sigma_{i_h})]
		\end{gathered}
		\end{equation}
	El cociente $\Sigma_n /(\Sigma_{i_1}\times \cdots \times \Sigma_{i_h})$
	es el grupo de los $(i_1,\ldots,i_h)$ intercambios de $\Sigma_n$,
	el cual se denota $Sh(i_1,\ldots,i_h)$.  Recuerde que un $(i_1,\ldots,i_h)$ intercambio,
	donde $i_1+\dots+i_h=n$, es un elemento de $\Sigma_n$ que envía
	$(1,\ldots,n)$ a $(\mu^1_{1},\ldots,\mu^1_{i_1},\ldots,\mu^h_{1},\ldots,\mu^h_{i_h})$
	tal que $\mu^j_{1}<\ldots<\mu^j_{i_j}$ para todo $1\leq j\leq h$.
	Así, \ref{fm-tdlsf} se escribe,
		\begin{equation}\label{fm-equi-ope-sigmak}
		\begin{gathered}
		M(i_1)\otimes \cdots \otimes M(i_h)\otimes \field[Sh(i_1,\ldots,i_h)]
		\end{gathered}
		\end{equation}
	El cual junto con la parte $M(h)$ nos da la siguiente expresión.
		\begin{equation}\label{fm-equi-operads}
		\begin{gathered}
		M(h)\otimes_{\Sigma_h}  M(i_1)\otimes \cdots \otimes M(i_h)\otimes \field[Sh(i_1,\ldots,i_h)] 
		\end{gathered}
		\end{equation}
	Nuestro operads libre necesitará esta construcción para cualquier $n$ y todas las posibles sumas 
	$i_1\cdots+i_h=n$, es decir, necesitamos considerar la suma directa,
		\begin{equation}\label{fm-MM}
		\begin{gathered}
		\bigoplus_{n\geq 0} \bigoplus_{h\geq 0} M(h)\otimes_{\Sigma_h} \left( \bigoplus_{i_1+\cdots+i_h=n}
		M(i_1)\otimes \cdots \otimes M(i_h)\otimes \field[Sh(i_1,\ldots,i_h)] \right)
		\end{gathered}
		\end{equation}
	Esta expresión representa la primera etapa de todas las posibles composiciones formales
	entre los elementos de $M$ cuando son interpretados como aplicaciones.
	En la siguiente etapa de composiciones se deben de considerar
	cuando cada $M(i_j)$ en \ref{fm-MM} viene
	de otra composición arbitraria y así sucesivamente, una cantidad finita de pasos.
	Para poder manejar todas los posibles niveles de composiciones
	necesitamos introducir algunas operaciones sobre los $\mathbb{S}$-módulos.

\begin{definition}
\label{df-tensor-product-symme-seq}
\index{tensorproductsmod@Tensor product of $\mathbb{S}$-modules}
	Sean $M$ y $N$ $\mathbb{S}$-módulos.
	Se define el producto tensorial de $M$ y $N$ como el $\mathbb{S}$-módulo $M\otimes N$ dado por la fórmula,
		\begin{equation}
		\begin{gathered}
		(M\otimes N)(n)=\bigoplus_{i+j=n} M(i)\otimes M(j)\otimes \field[Sh(i,j)]
		\end{gathered}
		\end{equation}
\end{definition}

\begin{proposition}
\label{prop-s-mod-properties}
	El producto tensorial de $\mathbb{S}$-módulos es asociativo y
	para cada $\mathbb{S}$-módulo $M$ satisface $M=M\otimes \field =\field \otimes M$,
	donde  $\field$ se ve como un $\mathbb{S}$-módulo concentrado en aridad $0$.
\end{proposition}

\begin{proof}
	Let $M$, $N$ and $P$ be $\mathbb{S}$ módulos. 
		\begin{align*}
		((M\otimes N)\otimes P)(n)&= \bigoplus_{i+j=n}(M\otimes N)(i)\otimes P(j)\otimes \field[Sh(i,j)]\\
		&= \bigoplus_{i+j=n}\bigoplus_{r+s=i}M(r)\otimes N(s)\otimes \field[Sh(r,s)]\otimes P(j)\otimes \field[Sh(i,j)]\\
		&=\bigoplus_{i+j=n}\bigoplus_{r+s=i} \left( M(r)\otimes N(s)\otimes_{\Sigma_r\times \Sigma_s} \field[\Sigma_i] \right)
		\otimes_{\Sigma_i\times \Sigma_j}  P(j)\otimes \field[\Sigma_n] \\
		&=\bigoplus_{r+s+j=n} M(r)\otimes N(s)\otimes P(j) \otimes_{\Sigma_r\times \Sigma_s\times \Sigma_j} \field[\Sigma_n]\\
		&=\bigoplus_{r+i=n}\bigoplus_{s+j=i}M(r) \otimes
		\left( N(s)\otimes P(j)\otimes_{\Sigma_s\times \Sigma_j} \field[\Sigma_i]\right) 
		\otimes_{\Sigma_r\times \Sigma_i}  \field[\Sigma_n]\\
		&=(M\otimes (N\otimes P))(n)
		\end{align*}
	Se deja el resto de la prueba al lector.
\end{proof}

\begin{remark}
	Note que en la fórmula \ref{fm-MM} se tiene,
		\begin{equation}\label{fm-decompo-Nk}
		\begin{gathered}
		\bigoplus_{i_1+\cdots+i_h=n} M(i_1)\otimes \cdots \otimes M(i_h)\otimes \field[Sh(i_1,\ldots,i_h)]=M^{\otimes h} 
		\end{gathered}
		\end{equation}
	donde $M^{\otimes h}$ es $h$ veces el producto tensorial de $\mathbb{S}$-módulos.
\end{remark}

\begin{definition}
\label{df-composition-symmetric-seq}
\index{compositionsmodules@Composition of $\mathbb{S}$-modules}
	Sean $M$ y $N$ $\mathbb{S}$-módulos. Se define la composición de $M$ con $N$ como el $\mathbb{S}$-módulo,
		\begin{equation}\label{fm-comp-S-mod}
		\begin{gathered}
		M\circ N= \bigoplus_{h\geq 0} M(h)\otimes_{\Sigma_h} N^{\otimes h}
		\end{gathered}
		\end{equation}
\end{definition}

\begin{remark}
	La fórmula \ref{fm-MM} puede ser escrita,
		\begin{align}
		&\bigoplus_{h\geq 0} M(h)\otimes_{\Sigma_h} \left( \bigoplus_{n\geq 0}\bigoplus_{i_1+\cdots+i_h=n}
		M(i_1)\otimes \cdots \otimes M(i_h)\otimes \field[Sh(i_1,\ldots,i_h)] \right) \nonumber\\ 
		&=\bigoplus_{h\geq 0} M(h)\otimes_{\Sigma_h} (M^{\otimes h}) =M\circ M
		\end{align}
	La composición de $M\circ M$ representa la primera etapa de composiciones formales y la expresión
	$M^{\circ h}$ puede ser usada para representar $h$ etapas de composiciones formales.
\end{remark}

\begin{proposition}
	Sean $f:M\to N$ y $f':M'\to N'$ morfismos de $\mathbb{S}$-módulos,
	entonces el morfismo dado por $(f\circ f')(x\otimes y_1\otimes \cdots \otimes y_h)=f(x)\otimes 
	f'(y_1)\otimes \cdots \otimes f'(y_h)$ es un morfismo de $\mathbb{S}$-módulos
	de $M\circ M'$ a $N\circ N'$. 
\end{proposition}\qed

	Los $\mathbb{S}$-módulos pueden ser identificados con endofuntores de 
	la categoría $\dga$-Mod de tal manera que la composición de $\dga$-módulos
	coincida con la composición de funtores (ver \cite{loday2012algebraic}, \S5).
	Este tipo de funtores se llaman funtores de Schur.
\index{shurfunctor@Schur functor}
	La siguiente proposición es una consecuencias de esta identificación.

\begin{proposition}
\label{prop-symmetric-seq-monoidal-cat}
	La categoría de $\mathbb{S}$-módulos con la composición $\circ$
 	y  $I=(0,\field,0,\ldots)$ es una categoría monoidal.
\end{proposition}\qed

De hecho, los operads son instancias de monoides sobre la categoría de $\mathbb{S}$-módulos.

\begin{proposition}
\label{prop-operads-monoids}
	Todo operad determina un monoide en $\mathbb{S}$-Mod y viceversa.
\end{proposition}

\begin{proof}
	Observe que una aplicación de $\mathbb{S}$-módulos $\eta:I\to M$ es no cero
	solamente en aridad $1$, entonces determina una aplicación $\eta:\field  \to M(1)$ y viceversa.
	Un morfismo $\mu:M\circ M\to M$ de $\mathbb{S}$-módulos en aridad $n$ está
	dado por un morfismo equivariante de $\dga$-módulos $\mu_n$,
		\begin{align}
		\mu_n:\bigoplus_{h\geq 0} \bigoplus_{i_1+\cdots+i_h=n} M(h)\otimes_{\Sigma_h} \left( 
		M(i_1)\otimes \cdots \otimes M(i_h)\otimes \field[Sh(i_1,\ldots,i_h)] \right) \to M(n)
		\end{align}
	el cual está determinado por la colección de morfismos equivariantes,
		\begin{align}
		\gamma:M(h)\otimes_{\Sigma_h} 
		\left(M(i_1)\otimes \cdots \otimes M(i_h)\otimes \field[Sh(i_1,\ldots,i_h)] \right) \to M(n)
		\end{align}
	y cada morfismo $\gamma$ está caracterizado como un morfismo
		\begin{align}
		\gamma:M(h)\otimes M(i_1)\otimes \cdots \otimes M(i_h) \to M(n)
		\end{align}
	que satisface las condiciones de equivarianza \ref{dg-equiv-1-op} y \ref{dg-equiv-2-op}.
	\end{proof}

Antes de la construcción de operads libres se necesita una operación más de $\mathbb{S}$-módulos.

\begin{definition}
\label{df-symme-seq-direct-sum}
\index{directsumsmodules@Direct sum of $\mathbb{S}$-modules}
	Sean $M$ y $N$ $\mathbb{S}$-módulos. Se define la suma directa de $M$ y $N$ mediante la fórmula,
		\begin{equation}
		\begin{gathered}
		(M\oplus N)(n)=M(n)\oplus N(n)
		\end{gathered}
		\end{equation}
\end{definition}

\begin{theorem}
\label{prop-adjunction-operads}
	El funtor de olvido $U:\mathcal{OP}\to \mathbb{S}$-Mod tiene un funtor adjunto a la izquierda
	$F:\mathbb{S}$-Mod$\to \mathcal{OP}$. Al funtor $F$ se le llama el funtor de operads libre.
\end{theorem}

\begin{proof}
	Sea $M$ un $\mathbb{S}$-módulo. Por la proposición \ref{prop-operads-monoids}
	solo necesitamos exhibir el operad libre como un monoide $(F(M),\mu, \eta)$ en $\mathbb{S}$-Mod.
	Se describirá la construcción de $F(M)$, del producto $\mu$ y el neutro $\eta$, 
	así como la construcción de la unidad 
	y la counidad de la adjunción.
	Las verificaciones de que estos objetos satisfacen las propiedades requeridas no son complicadas y se dejan
	al lector (para más detalles ver $\S 5.4$ en \cite{loday2012algebraic}).

	Primero, se construye inductivamente para cada $n$ un $\mathbb{S}$-módulo $F(M)_n$ de la siguiente manera.
		\begin{align}
		&F(M)_0= I \\
		&F(M)_1= I\oplus M\\
		&F(M)_2=I\oplus (M\circ (I\oplus M))=I\oplus (M\circ F(M)_1)\\
		&F(M)_{n+1}= I\oplus (M\circ F(M)_n)
		\end{align}
	Sea $i_0$ la inclusión de $I$ in $F(M)_1$. Usando la identidad $M=M\circ I$
	y los morfismos $1_{M}\circ i_0:M\circ I\to M\circ F(M)_1$
	obtenemos el morfismo $i_1=1_{I}\oplus (1_{M}\circ i_{0}):F(M)_1\to F(M)_2$.
	Repitiendo este proceso se obtienen los morfismos,
		\begin{align}
		i_n:F(M)_n\to F(M)_{n+1}
		\end{align}
	definidos por inducción con la fórmula $i_{n+1}=1_{I}\oplus (1_{M}\circ i_{n})$.
	Los $\mathbb{S}$-módulos $F(M)_{n}$ codifican todas las posibles $n$ etapas
	de composiciones de los elementos de $M$. Para poder colocar junta toda esta información
	en un solo $\mathbb{S}$-módulo, se toma el colímite sobre el diagrama
	determinado por los morfismos $i_j$.
		\begin{equation}
		\begin{gathered}
		F(M)=\underset{n}{\text{colím\,}} F(M)_n
		\end{gathered}
		\end{equation}
	El diferencial de $F(M)$ es la extensión evidente del diferencial de $M$.

	Ahora necesitamos definir el producto $\mu$ y el neutro $\eta$ para $F(M)$. 
	El neutro está dada por la inclusión $\eta:I\to F(M)$.
	El morfismo $\mu:F(M)\circ F(M)\to F(M)$ está determinado por una colección
	de aplicaciones $\mu_{n,m}:F(M)_n\circ F(M)_m\to F(M)_{n+m}$ definidas por inducción
	sobre $n$, tomando $\mu_{0,m}=1_{F(M)_m}$
	y para $n>0$, $\mu_{n,m}$ se define como  la composición,
		\begin{align*}
		F(M)_n\circ F(M)_m &= (I\oplus M\circ F(M)_{n-1})\circ F(M)_m \\
		&\overset{\cong}{\longrightarrow}  F(M)_m  \oplus (M\circ F(M)_{n-1})\circ F(M)_m\\
		&\overset{\cong}{\longrightarrow}  F(M)_m  \oplus M\circ (F(M)_{n-1}\circ F(M)_m)\\
		&\overset{\underrightarrow{\,\,\,1\oplus 1\circ \mu_{n-1,m}\,\,\,\,}}{} \,F(M)_m\oplus M\circ F(M)_{n+m-1}\\
		&\overset{\underrightarrow{\,\,\,i+i'\,\,\,\,}}{}  F(M)_{n+m}
		\end{align*}
	donde $i$ es la inclusión de $F(M)_m$ en $F(M)_{n+m}$, y $i'$
	es la inclusión de $F(M)_{n+m-1}$ como el segundo factor de $F(M)_{n+m}$.

	Sea $\mathcal{P}$ un operad, la counidad $\epsilon:FU\to 1$ de la adjunción está determinada
	por los morfismos $\epsilon_n:FU(\mathcal{P})_n\to \mathcal{P}$  definidos por inducción de la siguiente manera.
	$\epsilon_0:I\to \mathcal{P}$ está determinado por la unidad $\eta$ de $\mathcal{P}$,
	$\epsilon_1=\eta + 1:I\oplus U\mathcal{P}\to \mathcal{P}$ y
	$\epsilon_{n+1}=\eta+\gamma(1\circ \epsilon_{n}):FU(P)_n=I\oplus (U\mathcal{P}\circ UF(P)_{n})\to \mathcal{P}$.
	Finalmente, para $M\in \mathbb{S}$-Mod, la unidad de la adjunción $\eta:1\to UF$, está determinada por las
	inclusiones en el segundo factor $M\circ F(M)_{n-1}\to F(M)_n$.
\end{proof}

\label{rmk-description-adjuntion-op}
	En resumen, la adjunción de la proposición \ref{prop-adjunction-operads},
	define para cada operad $\mathcal{P}$ y $\mathbb{S}$-module $M$, la biyección natural,
		\begin{equation}
		\begin{gathered}
		\theta: \mathcal{OP}(F(M),\mathcal{P})\to \mathbb{S}\text{-Mod}(M,U(\mathcal{P})).
		\end{gathered}
		\end{equation}
	La unidad y la counidad de la adjunción son denotadas $\eta$ y $\epsilon$, respectivamente.
	Para la unidad $\eta$ tenemos los morfismos,
		\begin{align}
		& \eta:1_{\mathbb{S}\text{-Mod}}\to UF, \text{ y } \\
		& \eta_M:M\to UF(M).
		\end{align}
	Para la counidad $\epsilon$ tenemos,
		\begin{align}
		& \epsilon:FU \to 1_{\mathcal{OP}}, \text{ y } \\
		& \epsilon_\mathcal{P}:FU(\mathcal{P})\to \mathcal{P}.
		\end{align}

\section{Relaciones entre adjunciones}

\begin{definition}
\label{df-NSoperads}
\index{nonsymmetricoperad@Non symmetric operad}
\index{categorynonsymmetricoperads@Non symmetric operads category $n\mathcal{OP}$}
	Un operad no simétrico es definido como un operad simétrico
	pero sin las acciones de los grupos simétricos.
	La categoría de los operads no simétricos es denotada $n\mathcal{OP}$
\end{definition}

\begin{definition}
\label{df-N-sequences}
\index{nmodules@$\mathbb{N}$-module}
\index{categorynmodules@$\mathbb{N}$-modules category $\mathbb{N}$-Mod}
	Un $\mathbb{N}$-módulo es un funtor contravariante del grupoide $\mathfrak{N}$
	con objetos el conjunto $[0]=\emptyset$ y $[n]=\{1,..,n\}$ for $n>0$, y morfismos
	las aplicaciones identidad, a la categoría
	$\dga$-Mod. La categoría de $\mathbb{N}$-módulos se denota $\mathbb{N}$-Mod.
\end{definition}

\begin{definition}
\label{df-sym-seq-forget}
	Sea $G$ el funtor de olvido de $\mathbb{S}$-Mod a $\mathbb{N}$-Mod.
\end{definition}

\begin{theorem}
\label{prop-adj-sym-senq}
\index{forgetfulfunctorgsmodnmod@Forgetful functor $G:\mathbb{S}\text{-Mod}\to \mathbb{N}\text{-Mod}$ }
	El funtor de olvido $G:\mathbb{S}\text{-Mod}\to \mathbb{N}\text{-Mod}$ 
	tiene una adjunta a la izquierda $H:\mathbb{N}\text{-Mod} \to \mathbb{S}\text{-Mod}$, llamada el funtor
	libre de $\mathbb{S}$-módulos.
\index{freesmodulefunctor@Free $\mathbb{S}$-module functor}
\end{theorem}

\begin{proof}
	Para $M\in \mathbb{N}\text{-Mod}$, $H(M)$ es definido como el $\mathbb{S}$ módulo
	con componentes dadas por $M(n)\otimes \field[\Sigma_n]$, para cada $n$. La verificación que satisface las
	propiedades es directa.
\end{proof}

\begin{definition}
\label{df-nOp-Seq-forget}
\index{forgetfulfunctornonsymmopenmod@Forgetful functor $nU:n\mathcal{OP}\to \mathbb{N}\text{-Mod}$}
	Sea $\mathfrak{U}$ el funtor de olvido de la categoría $n\mathcal{OP}$ a $\mathbb{N}\text{-Mod}$.
\end{definition}

\begin{theorem}
\label{prop-adjunction-NSoperads}
\index{freenonsymmetricoperadfunctor@Free non symmetric operad functor $nF$}
	El funtor de olvida $nU:n\mathcal{OP}\to \mathbb{N}\text{-Mod}$
	tiene una adjunta a la izquierda $nF:\mathbb{N}\text{-Mod}\to n\mathcal{OP}$,
	llamada el funtor libre de operads no simétricos.
\end{theorem}

\begin{proof}
	El caso no simétrico es similar al caso simétrico, pero sin las consideraciones sobre las acciones de $\Sigma_n$.
	Para un $\mathbb{N}$-módulo $N$, con tal de construir un operad no simétrico libre sobre $N$
	necesitamos una suma directa, un producto tensorial y una composición de $\mathbb{N}$-módulos. Estas son
	definidas como,
		\begin{align}
		&(N\oplus E)(n)=N(n)\oplus E(n),\\
		&(N\otimes E)(n)=\bigoplus_{i+j=n} N(i)\otimes E(j) \text{ y }\\
		&(N\circ E)=\bigoplus_{h\geq 0}N(k)\otimes E^{\otimes h}. 
		\end{align}
	Note que,
		\begin{align}
		N\circ E=\bigoplus_{h\geq 0} \bigoplus_{n\geq 0} 
		\bigoplus_{i_1+\cdots+i_h=n}N(h)\otimes E(i_1)\otimes \cdots \otimes E(i_h)
		\end{align}
	Como en el caso simétrico, los operads no simétricos son monoides 
	sobre la categoría monoidal de $\mathbb{N}$-módulos, donde la estructura
	monoidal está dada por la composición. Los pasos
	para la construcción de $nF$ son los mismos que en el caso simétrico.
\end{proof}

\begin{definition}
\label{df-forget-op-nop}
\index{forgetfulfunctoropnonop@Forgetful functor $\mathcal{G}:\mathcal{OP}\to n\mathcal{OP}$}
	Sea $\mathcal{G}$ el funtor de olvido de la categoría de operads simétricos
	$\mathcal{OP}$ a la categoría de operads no simétricos $n\mathcal{OP}$. 
\end{definition}

\begin{theorem}
\label{prop-adj-OP-nOP}
El funtor $\mathcal{G}:\mathcal{OP}\to n\mathcal{OP}$
tiene un funtor adjunto a la izquierda $\mathcal{H}:n\mathcal{OP}\to \mathcal{OP}$.
\end{theorem}

\begin{proof}
	Para un operad no simétrico $\mathcal{P}$ su operad simétrico asociado está dado por 
	$\mathcal{P}\otimes \field[\Sigma]$.  Las verificaciones no tienen dificultad y se dejan al lector.
\end{proof}

\begin{theorem}
\label{prop-rela-forget-free-functors}
	Las relaciones entre estos funtores de olvido y sus asociados funtores libres, se reúnen en 
	los siguiente diagramas conmutativos.
		\begin{equation}
		\begin{gathered}
		\xymatrix@C=3pc@R=3pc{
		\mathcal{OP} \ar[r]^-{\mathcal{G}} \ar[d]_-{U} & n\mathcal{OP} \ar[d]^-{nU}\\
		\mathbb{S}\text{-Mod} \ar[r]_-{G} & \mathbb{N}\text{-Mod}
		}
		\xymatrix@C=3pc@R=3pc{
		&\mathcal{OP} \ar@{<-}[r]^-{\mathcal{H}} \ar@{<-}[d]_-{F} & n\mathcal{OP} \ar@{<-}[d]^-{nF}\\
		&\mathbb{S}\text{-Mod} \ar@{<-}[r]_-{H} & \mathbb{N}\text{-Mod}
		}
		\end{gathered}	
		\end{equation}
\end{theorem}

\begin{proof}
La conmutatividad es inmediata por el hecho de que la composición de funtores adjuntos es de nuevo una adjunción, y
por la unicidad de la adjunción, sus imágenes son isomorfas.
\end{proof}

\label{remark-noperads}
	El diagrama conmutativo de la proposición \ref{prop-rela-forget-free-functors}
	sugiere que se puede usar la construcción de operads no simétricos
	para describir los operads libres sobre secuencias simétricas en las cuales las acciones
	de los grupos simétricos son libres, es decir,  para un $\mathbb{S}$-módulo $M$
	el operad libre $F(M)$ podría ser interpretado por medio del operad $(\mathcal{H}\circ nF\circ G)(M)$.
	Lo cual corresponde al caso donde los $\mathbb{S}$-módulos tiene como 
	componentes bar resoluciones $\Sigma_n$-libres. Este tipo de $\mathbb{S}$-módulos 
	son usados para describir una $E_\infty$-coalgebra estructura  en los complejos de cadenas asociados 
	a conjuntos simpliciales (ver \cite{sanchez-jesus-thesis-2016}).

\begin{definition}
\label{def-coequalizador-de-s-modulos}
Sea $\Psi:\mathbb{S}\text{-Mod}\to \mathbb{S}\text{-Mod}$ el funtor que convierte las acciones
por los grupos simétricos de un $\mathbb{S}$-módulo en triviales, es decir,
		\begin{equation}
		\begin{gathered}
\Psi(M)(n)=\underset{\sigma \in \Sigma_n}{\text{coeq}} \left( M(n) \underset{\longrightarrow}{\sigma} M(n)\right)
		\end{gathered}	
		\end{equation}

\end{definition}

\begin{example}Se tiene la siguiente relación entre los operads 
de los ejemplos \ref{N-operad} y \ref{M-operad}: 

		\begin{equation}
		\begin{gathered}
\Psi (U(\mathcal{M}))=U(\mathcal{N}).
		\end{gathered}	
		\end{equation}

\end{example}

\begin{corollary}
\label{cor-isomorfismo-del-funtor-libre}
En la subcategoría de $\mathbb{S}$-módulos donde la acción de los grupos simétricos es libre,
la restricción del funtor de operads libres $F$  es isomorfo al funtor $\mathcal{H}\circ nF\circ G \circ \Psi$.
\end{corollary}

\begin{proof}
La prueba es inmediata y se deja al lector.
\end{proof}

La caracterización dada por el corolario \ref{cor-isomorfismo-del-funtor-libre} es útil
a la hora de imaginar operads libres, debido a que en el caso no simétrico los operads
libres pueden ser visualizados fácilmente por árboles planos (ver \cite{markl2007operads}).

\bibliographystyle{amsplain}
\bibliography{principal-bibliography}
\end{document}